\documentclass[12pt]{article}

\usepackage{authblk}

\usepackage{amsfonts}
\usepackage{graphicx}
\usepackage{tabularx}
\usepackage{array}
\usepackage{hyperref}
\usepackage[usenames,dvipsnames]{color}
\usepackage{comment}
\usepackage{amsmath}
\usepackage{amsthm}
\usepackage{amssymb}
\usepackage{fullpage}
\usepackage[dvipsnames]{xcolor}
\usepackage{tikz}
\usepackage{tkz-graph}
\usepackage{listings}
\usepackage{dsfont}

\newtheorem{theorem}{Theorem}[section]
\newtheorem{proposition}[theorem]{Proposition}

\newtheorem{conjecture}[theorem]{Conjecture}
\newtheorem{question}[theorem]{Question}

\newtheorem{lemma}[theorem]{Lemma}
\newtheorem{corollary}[theorem]{Corollary}
\theoremstyle{definition}

\def\epsilon{\varepsilon}
\DeclareMathOperator{\dist}{dist}
\DeclareMathOperator{\BS}{\mathcal{B}}
\DeclareMathOperator{\cI}{\mathcal{I}}

\title{A Counterexample to a Conjecture of Lov\'{a}sz}

\author[1]{Alexander Clow\,}
\author[2]{Penny Haxell\,\thanks{Supported in part by an NSERC Discovery Grant, and in part by National Science
Foundation Grant No. DMS-1928930, while the author was in
residence at the Simons Laufer Mathematical Sciences Institute in
Berkeley, California, during the Spring 2025 semester.}}
\author[3]{Bojan Mohar\,\thanks{Supported in part by the NSERC Discovery Grant R832714 (Canada),
by the ERC Synergy grant (European Union, ERC, KARST, project number 101071836),
and by the Research Project N1-0218 of ARIS (Slovenia).}\thanks{On leave from:
FMF, Department of Mathematics, University of Ljubljana, Ljubljana, Slovenia.}}

\affil[1,3]{ \small{Department of Mathematics, Simon Fraser University}}
\affil[2]{ \small{Department of Combinatorics and Optimization, University of Waterloo}}

\date{}

\begin{document}

\maketitle

\begin{abstract}
   In 1975 Lov\'{a}sz conjectured that every $r$-partite, $r$-uniform hypergraph contains $r-1$ vertices whose deletion reduces the matching number. If true, this statement would imply a well-known conjecture of Ryser from 1971, which states that every $r$-partite, $r$-uniform hypergraph has a vertex cover of size at most $r-1$ times its matching number.
   When $r=2$, Ryser's conjecture is simply K\H{o}nig's theorem, and the conjecture of Lov\'asz is an immediate corollary. Ryser's conjecture for $r=3$ was proven by Aharoni in 2001, and remains open for all $r\geq 4$.
   
   Here we show that the conjecture of Lov\'asz is false in the case $r=3$. Our counterexample is the line hypergraph of the Biggs-Smith graph, a highly symmetric cubic graph on 102 vertices.
\end{abstract}

\section{Introduction}

An \emph{$r$-partite $r$-uniform hypergraph} $\mathcal{H}$ consists of a \emph{vertex set} $V=V(\mathcal{H})$ that is partitioned into $r$ \emph{parts} $V_1,\ldots,V_r$, and a set $E = E(\mathcal{H})$ of \emph{edges}, where each edge contains exactly one vertex from each part. This definition naturally generalizes the notion of a bipartite graph, which is a 2-partite 2-uniform hypergraph.

A \emph{matching} in a hypergraph is a set of pairwise disjoint edges, and the
\emph{matching number} of $\mathcal{H}$, denoted $\nu(\mathcal{H})$, is the size of a largest matching in $\mathcal{H}$.
A \emph{vertex cover} of $\mathcal{H}$ is a set of vertices that intersects every edge of $\mathcal{H}$, and we denote by $\tau(\mathcal{H})$
the \emph{vertex cover number} of $\mathcal{H}$, which is the size of a smallest vertex cover in $\mathcal{H}$. These definitions imply that for any $r$-uniform hypergraph we have
$$
\nu(\mathcal{H}) \leq \tau(\mathcal{H}) \leq r\nu(\mathcal{H}).
$$
One can easily verify that both inequalities are tight for general hypergraphs.
However, it is natural to consider whether these bounds might be improved for special classes of hypergraphs. 
The following conjecture appears in the PhD thesis of J. R. Henderson from 1971\cite{henderson1971permutation}, where it is attributed to his PhD supervisor Herbert Ryser. 

\begin{conjecture}[Ryser's Conjecture]\label{conj:ryser}
    Every $r$-partite $r$-uniform hypergraph $\mathcal{H}$ satisfies
    $$
    \tau(\mathcal{H}) \leq (r-1)\nu(\mathcal{H}).
    $$
\end{conjecture}

In 1975 Lov\'{a}sz \cite{lovasz1975minimax} conjectured the stronger statement that
not only do we have a vertex cover of size $(r-1)\nu(\mathcal{H})$, but we can obtain one by repeatedly reducing the matching number with the removal of $r-1$ vertices.

\begin{conjecture}[Lov\'{a}sz 1975]\label{conj:lovasz}
    Every $r$-partite, $r$-uniform hypergraph with at least one edge contains $r-1$ vertices whose deletion reduces the matching number.
\end{conjecture}

Both of these conjectures have proven to be extremely difficult.
When $r=2$, Conjecture~\ref{conj:ryser} is precisely K\H{o}nig's theorem, and Conjecture~\ref{conj:lovasz} is an immediate corollary.
Ryser's conjecture in the case $r=3$ was proven by Aharoni \cite{aharoni2001ryser} in 2001, using a Hall-type theorem for hypergraphs earlier
proven by Aharoni and Haxell \cite{aharoni2000hall}.
More recently, Haxell, Narins, and Szab\'{o} \cite{haxell2018extremal} characterized the $3$-partite $3$-uniform hypergraphs $\mathcal{H}$ with $\tau(\mathcal{H})=2\nu(\mathcal{H})$.

Ryser's conjecture remains open for all $r \geq 4$.
There are a number of partial results, for example
Haxell and Scott \cite{haxell2012ryser} proved that for $r \leq 5$ there exists an $\epsilon>0$ such that $\tau(\mathcal{H}) \leq (r-\epsilon)\nu(\mathcal{H})$ for any $r$-partite $r$-uniform hypergraph.
The conjecture has also been proven for hypergraphs with $\nu(\mathcal{H}) = 1$ when $r\leq 5$ by Tuza \cite{tuza1979some,tuza1983ryser}.
The case $r\geq 6$ is still open even for hypergraphs with $\nu(\mathcal{H}) = 1$.
More recently, Franceti\'c, Herke, McKay, and Wanless \cite{francetic2017ryser} proved Ryser’s Conjecture for linear intersecting hypergraphs (i.e. hypergraphs where any two edges intersect in exactly one vertex) when $r\leq 9$.
These results were later expanded to $t$-intersecting hypergraphs (i.e. hypergraphs where any two edges intersect in at least $t$ vertices) by Bustamante and Stein \cite{bustamante2018monochromatic} and independently by Kir\'{a}ly and T\'{o}thm\'{e}r\'{e}sz \cite{francetic2017ryser}.
The most recent progress on the problem for $t$-intersecting hypergraphs is by Bishnoi, Das, Morris, and Szab\'{o} \cite{bishnoi2021ryser}, who obtained tight bounds on the vertex cover number of some classes of $r$-uniform $t$-intersecting hypergraphs.
Fractional versions of the conjecture have also been studied, in particular it was shown by F\H{u}redi \cite{furedi1981maximum} that $\tau^* \leq (r-1)\nu$, and shown by Lov\'{a}sz \cite{lovasz1974minimax} that $\tau\leq \frac{r}{2}\nu^*$, where $\tau^*$ and $\nu^*$ are the fractional vertex cover and matching numbers, respectively.

The original source~\cite{lovasz1975minimax} for Conjecture~\ref{conj:lovasz} is in Hungarian, but the conjecture has been widely known since its appearance in the influential 1988 survey of F{\"u}redi~\cite{furedi1988matchings} on matchings and covers in hypergraphs. However, not much work has been published on the problem over the years.
To the authors' knowledge, the most significant progress comes from work by Haxell, Narins, and Szab\'{o} \cite{haxell2018extremal}, who showed every $3$-partite $3$-uniform hypergraph $\mathcal{H}$ with $\tau(\mathcal{H})=2\nu(\mathcal{H})$ satisfies Conjecture~\ref{conj:lovasz}.

The primary contribution of this paper is to prove that Conjecture~\ref{conj:lovasz} is false.

\begin{theorem}\label{Thm: Lovasz Conj False}
    There exists a $3$-partite, $3$-uniform hypergraph such that deleting any pair of vertices does not reduce the matching number.
\end{theorem}

For a graph $G = (V,E)$, the \emph{line hypergraph} $\mathcal{L}(G)$ of $G$ is the hypergraph $\mathcal{L} = (V',E')$ where $V' = E$ and $e_v = \{e_1,\dots, e_k\} \in E'$ if and only if $\{e_1,\dots, e_k\}$ is the set of all edges incident to a vertex $v \in V$. Then $\nu(\mathcal{L}(G))=\alpha(G)$, where $\alpha(G)$ denotes the \emph{independence number} of $G$, i.e. the size of a largest independent set in $G$.
We will prove that the line hypergraph of the graph shown in Figure~\ref{Fig:Basic Biggs-Smith} is a counterexample to Conjecture~\ref{conj:lovasz}.
This graph is a cubic graph with $102$ vertices and is known as the Biggs-Smith graph. A precise description of the graph is given in Section~\ref{sec:BSG}, where we also detail its many special properties that will be needed for our argument.  
Principal among these is the fact that it is distance-transitive (see Section~\ref{sec:BSG}), which was shown by Biggs and Smith in Section~1 and Section~4 of \cite{biggs1971trivalent}.
In fact, as proved in~\cite{biggs1971trivalent}, only $12$ cubic distance-transitive graphs exist (up to isomorphism),
and this graph was first discovered by Biggs and Smith as part of the classification of cubic distance-transitive graphs.

\begin{figure}[htb]
\begin{center}  
\scalebox{0.235}{
\begin{tikzpicture}
\GraphInit[vstyle=Simple]
\Vertex[L=\hbox{$0$},x=5.7586cm,y=6.1162cm]{v0}
\Vertex[L=\hbox{$1$},x=5.177cm,y=7.5245cm]{v1}
\Vertex[L=\hbox{$2$},x=4.2953cm,y=8.5227cm]{v2}
\Vertex[L=\hbox{$3$},x=3.2326cm,y=8.976cm]{v3}
\Vertex[L=\hbox{$4$},x=2.1325cm,y=8.8231cm]{v4}
\Vertex[L=\hbox{$5$},x=1.1434cm,y=8.0848cm]{v5}
\Vertex[L=\hbox{$6$},x=0.3991cm,y=6.8607cm]{v6}
\Vertex[L=\hbox{$7$},x=0.0cm,y=5.3162cm]{v7}
\Vertex[L=\hbox{$8$},x=0.0cm,y=3.6598cm]{v8}
\Vertex[L=\hbox{$9$},x=0.3991cm,y=2.1152cm]{v9}
\Vertex[L=\hbox{$10$},x=1.1434cm,y=0.8912cm]{v10}
\Vertex[L=\hbox{$11$},x=2.1325cm,y=0.1528cm]{v11}
\Vertex[L=\hbox{$12$},x=3.2326cm,y=0.0cm]{v12}
\Vertex[L=\hbox{$13$},x=4.2953cm,y=0.4533cm]{v13}
\Vertex[L=\hbox{$14$},x=5.177cm,y=1.4515cm]{v14}
\Vertex[L=\hbox{$15$},x=5.7586cm,y=2.8598cm]{v15}
\Vertex[L=\hbox{$16$},x=5.9616cm,y=4.488cm]{v16}
\Vertex[L=\hbox{$17$},x=9.9701cm,y=12.0cm]{v17}
\Vertex[L=\hbox{$18$},x=5.9616cm,y=19.512cm]{v18}
\Vertex[L=\hbox{$19$},x=3.2326cm,y=24.0cm]{v19}
\Vertex[L=\hbox{$20$},x=8.1508cm,y=14.992cm]{v20}
\Vertex[L=\hbox{$21$},x=16.1679cm,y=14.992cm]{v21}
\Vertex[L=\hbox{$22$},x=21.271cm,y=8.976cm]{v22}
\Vertex[L=\hbox{$23$},x=23.2154cm,y=7.5245cm]{v23}
\Vertex[L=\hbox{$24$},x=17.4641cm,y=14.0243cm]{v24}
\Vertex[L=\hbox{$25$},x=9.447cm,y=14.0243cm]{v25}
\Vertex[L=\hbox{$26$},x=5.177cm,y=22.5485cm]{v26}
\Vertex[L=\hbox{$27$},x=5.177cm,y=16.4755cm]{v27}
\Vertex[L=\hbox{$28$},x=9.447cm,y=9.9757cm]{v28}
\Vertex[L=\hbox{$29$},x=17.4641cm,y=9.9757cm]{v29}
\Vertex[L=\hbox{$30$},x=23.2154cm,y=16.4755cm]{v30}
\Vertex[L=\hbox{$31$},x=19.1818cm,y=23.1088cm]{v31}
\Vertex[L=\hbox{$32$},x=22.3337cm,y=15.4773cm]{v32}
\Vertex[L=\hbox{$33$},x=16.8763cm,y=9.3102cm]{v33}
\Vertex[L=\hbox{$34$},x=8.8593cm,y=9.3102cm]{v34}
\Vertex[L=\hbox{$35$},x=4.2953cm,y=15.4773cm]{v35}
\Vertex[L=\hbox{$36$},x=5.7586cm,y=21.1402cm]{v36}
\Vertex[L=\hbox{$37$},x=2.1325cm,y=23.8472cm]{v37}
\Vertex[L=\hbox{$38$},x=7.4174cm,y=14.8901cm]{v38}
\Vertex[L=\hbox{$39$},x=15.4344cm,y=14.8901cm]{v39}
\Vertex[L=\hbox{$40$},x=20.1709cm,y=23.8472cm]{v40}
\Vertex[L=\hbox{$41$},x=21.271cm,y=15.024cm]{v41}
\Vertex[L=\hbox{$42$},x=21.271cm,y=24.0cm]{v42}
\Vertex[L=\hbox{$43$},x=20.1709cm,y=15.1769cm]{v43}
\Vertex[L=\hbox{$44$},x=15.4344cm,y=9.1099cm]{v44}
\Vertex[L=\hbox{$45$},x=7.4174cm,y=9.1099cm]{v45}
\Vertex[L=\hbox{$46$},x=2.1325cm,y=15.1769cm]{v46}
\Vertex[L=\hbox{$47$},x=0.0cm,y=20.3402cm]{v47}
\Vertex[L=\hbox{$48$},x=5.9957cm,y=12.5521cm]{v48}
\Vertex[L=\hbox{$49$},x=14.0128cm,y=12.5521cm]{v49}
\Vertex[L=\hbox{$50$},x=18.0384cm,y=5.3162cm]{v50}
\Vertex[L=\hbox{$51$},x=19.1818cm,y=8.0848cm]{v51}
\Vertex[L=\hbox{$52$},x=14.7751cm,y=14.3979cm]{v52}
\Vertex[L=\hbox{$53$},x=6.758cm,y=14.3979cm]{v53}
\Vertex[L=\hbox{$54$},x=1.1434cm,y=23.1088cm]{v54}
\Vertex[L=\hbox{$55$},x=0.3991cm,y=17.1393cm]{v55}
\Vertex[L=\hbox{$56$},x=6.2618cm,y=10.4182cm]{v56}
\Vertex[L=\hbox{$57$},x=14.2789cm,y=10.4182cm]{v57}
\Vertex[L=\hbox{$58$},x=18.4375cm,y=2.1152cm]{v58}
\Vertex[L=\hbox{$59$},x=20.1709cm,y=0.1528cm]{v59}
\Vertex[L=\hbox{$60$},x=22.3337cm,y=0.4533cm]{v60}
\Vertex[L=\hbox{$61$},x=23.797cm,y=2.8598cm]{v61}
\Vertex[L=\hbox{$62$},x=17.8519cm,y=10.9145cm]{v62}
\Vertex[L=\hbox{$63$},x=9.8348cm,y=10.9145cm]{v63}
\Vertex[L=\hbox{$64$},x=5.7586cm,y=17.8838cm]{v64}
\Vertex[L=\hbox{$65$},x=4.2953cm,y=23.5467cm]{v65}
\Vertex[L=\hbox{$66$},x=8.8593cm,y=14.6898cm]{v66}
\Vertex[L=\hbox{$67$},x=16.8763cm,y=14.6898cm]{v67}
\Vertex[L=\hbox{$68$},x=22.3337cm,y=23.5467cm]{v68}
\Vertex[L=\hbox{$69$},x=19.1818cm,y=15.9152cm]{v69}
\Vertex[L=\hbox{$70$},x=23.2154cm,y=22.5485cm]{v70}
\Vertex[L=\hbox{$71$},x=18.4375cm,y=17.1393cm]{v71}
\Vertex[L=\hbox{$72$},x=23.797cm,y=21.1402cm]{v72}
\Vertex[L=\hbox{$73$},x=18.0384cm,y=18.6838cm]{v73}
\Vertex[L=\hbox{$74$},x=14.0128cm,y=11.4479cm]{v74}
\Vertex[L=\hbox{$75$},x=5.9957cm,y=11.4479cm]{v75}
\Vertex[L=\hbox{$76$},x=0.0cm,y=18.6838cm]{v76}
\Vertex[L=\hbox{$77$},x=3.2326cm,y=15.024cm]{v77}
\Vertex[L=\hbox{$78$},x=8.1508cm,y=9.008cm]{v78}
\Vertex[L=\hbox{$79$},x=16.1679cm,y=9.008cm]{v79}
\Vertex[L=\hbox{$80$},x=21.271cm,y=0.0cm]{v80}
\Vertex[L=\hbox{$81$},x=23.2154cm,y=1.4515cm]{v81}
\Vertex[L=\hbox{$82$},x=24.0cm,y=4.488cm]{v82}
\Vertex[L=\hbox{$83$},x=17.9872cm,y=12.0cm]{v83}
\Vertex[L=\hbox{$84$},x=24.0cm,y=19.512cm]{v84}
\Vertex[L=\hbox{$85$},x=18.0384cm,y=20.3402cm]{v85}
\Vertex[L=\hbox{$86$},x=23.797cm,y=17.8838cm]{v86}
\Vertex[L=\hbox{$87$},x=18.4375cm,y=21.8848cm]{v87}
\Vertex[L=\hbox{$88$},x=14.2789cm,y=13.5818cm]{v88}
\Vertex[L=\hbox{$89$},x=6.2618cm,y=13.5818cm]{v89}
\Vertex[L=\hbox{$90$},x=0.3991cm,y=21.8848cm]{v90}
\Vertex[L=\hbox{$91$},x=1.1434cm,y=15.9152cm]{v91}
\Vertex[L=\hbox{$92$},x=6.758cm,y=9.6021cm]{v92}
\Vertex[L=\hbox{$93$},x=14.7751cm,y=9.6021cm]{v93}
\Vertex[L=\hbox{$94$},x=19.1818cm,y=0.8912cm]{v94}
\Vertex[L=\hbox{$95$},x=18.0384cm,y=3.6598cm]{v95}
\Vertex[L=\hbox{$96$},x=18.4375cm,y=6.8607cm]{v96}
\Vertex[L=\hbox{$97$},x=20.1709cm,y=8.8231cm]{v97}
\Vertex[L=\hbox{$98$},x=22.3337cm,y=8.5227cm]{v98}
\Vertex[L=\hbox{$99$},x=23.797cm,y=6.1162cm]{v99}
\Vertex[L=\hbox{$100$},x=17.8519cm,y=13.0855cm]{v100}
\Vertex[L=\hbox{$101$},x=9.8348cm,y=13.0855cm]{v101}
\Edge[](v0)(v16)
\Edge[](v0)(v1)
\Edge[](v0)(v101)
\Edge[](v1)(v2)
\Edge[](v1)(v25)
\Edge[](v2)(v66)
\Edge[](v2)(v3)
\Edge[](v3)(v4)
\Edge[](v3)(v20)
\Edge[](v4)(v5)
\Edge[](v4)(v38)
\Edge[](v5)(v53)
\Edge[](v5)(v6)
\Edge[](v6)(v7)
\Edge[](v6)(v89)
\Edge[](v7)(v48)
\Edge[](v7)(v8)
\Edge[](v8)(v9)
\Edge[](v8)(v75)
\Edge[](v9)(v56)
\Edge[](v9)(v10)
\Edge[](v10)(v11)
\Edge[](v10)(v92)
\Edge[](v11)(v12)
\Edge[](v11)(v45)
\Edge[](v12)(v13)
\Edge[](v12)(v78)
\Edge[](v13)(v34)
\Edge[](v13)(v14)
\Edge[](v14)(v28)
\Edge[](v14)(v15)
\Edge[](v15)(v16)
\Edge[](v15)(v63)
\Edge[](v16)(v17)
\Edge[](v17)(v18)
\Edge[](v17)(v83)
\Edge[](v18)(v19)
\Edge[](v18)(v77)
\Edge[](v19)(v20)
\Edge[](v19)(v47)
\Edge[](v20)(v21)
\Edge[](v21)(v22)
\Edge[](v21)(v42)
\Edge[](v22)(v51)
\Edge[](v22)(v23)
\Edge[](v23)(v82)
\Edge[](v23)(v24)
\Edge[](v24)(v70)
\Edge[](v24)(v25)
\Edge[](v25)(v26)
\Edge[](v26)(v54)
\Edge[](v26)(v27)
\Edge[](v27)(v91)
\Edge[](v27)(v28)
\Edge[](v28)(v29)
\Edge[](v29)(v81)
\Edge[](v29)(v30)
\Edge[](v30)(v87)
\Edge[](v30)(v31)
\Edge[](v31)(v32)
\Edge[](v31)(v52)
\Edge[](v32)(v33)
\Edge[](v32)(v40)
\Edge[](v33)(v34)
\Edge[](v33)(v60)
\Edge[](v34)(v35)
\Edge[](v35)(v36)
\Edge[](v35)(v55)
\Edge[](v36)(v37)
\Edge[](v36)(v101)
\Edge[](v37)(v38)
\Edge[](v37)(v76)
\Edge[](v38)(v39)
\Edge[](v39)(v97)
\Edge[](v39)(v40)
\Edge[](v40)(v41)
\Edge[](v41)(v42)
\Edge[](v41)(v79)
\Edge[](v42)(v43)
\Edge[](v43)(v68)
\Edge[](v43)(v44)
\Edge[](v44)(v59)
\Edge[](v44)(v45)
\Edge[](v45)(v46)
\Edge[](v46)(v64)
\Edge[](v46)(v47)
\Edge[](v47)(v48)
\Edge[](v48)(v49)
\Edge[](v49)(v50)
\Edge[](v49)(v85)
\Edge[](v50)(v51)
\Edge[](v50)(v58)
\Edge[](v51)(v52)
\Edge[](v52)(v53)
\Edge[](v53)(v54)
\Edge[](v54)(v55)
\Edge[](v55)(v56)
\Edge[](v56)(v57)
\Edge[](v57)(v71)
\Edge[](v57)(v58)
\Edge[](v58)(v59)
\Edge[](v59)(v60)
\Edge[](v60)(v61)
\Edge[](v61)(v99)
\Edge[](v61)(v62)
\Edge[](v62)(v86)
\Edge[](v62)(v63)
\Edge[](v63)(v64)
\Edge[](v64)(v65)
\Edge[](v65)(v66)
\Edge[](v65)(v90)
\Edge[](v66)(v67)
\Edge[](v67)(v98)
\Edge[](v67)(v68)
\Edge[](v68)(v69)
\Edge[](v69)(v70)
\Edge[](v69)(v93)
\Edge[](v70)(v71)
\Edge[](v71)(v72)
\Edge[](v72)(v100)
\Edge[](v72)(v73)
\Edge[](v73)(v84)
\Edge[](v73)(v74)
\Edge[](v74)(v75)
\Edge[](v74)(v95)
\Edge[](v75)(v76)
\Edge[](v76)(v77)
\Edge[](v77)(v78)
\Edge[](v78)(v79)
\Edge[](v79)(v80)
\Edge[](v80)(v81)
\Edge[](v80)(v94)
\Edge[](v81)(v82)
\Edge[](v82)(v83)
\Edge[](v83)(v84)
\Edge[](v84)(v85)
\Edge[](v85)(v86)
\Edge[](v86)(v87)
\Edge[](v87)(v88)
\Edge[](v88)(v96)
\Edge[](v88)(v89)
\Edge[](v89)(v90)
\Edge[](v90)(v91)
\Edge[](v91)(v92)
\Edge[](v92)(v93)
\Edge[](v93)(v94)
\Edge[](v94)(v95)
\Edge[](v95)(v96)
\Edge[](v96)(v97)
\Edge[](v97)(v98)
\Edge[](v98)(v99)
\Edge[](v99)(v100)
\Edge[](v100)(v101)
\end{tikzpicture}
}
\qquad \qquad 
\scalebox{1.35}{
\includegraphics[width=0.33\textwidth]{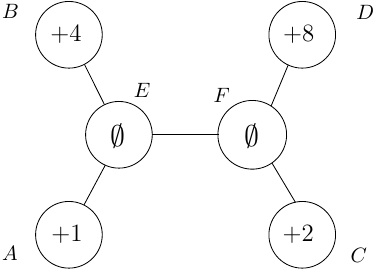}
}
\end{center}
\caption{A drawing of the Biggs-Smith graph and its representation as an order-17 expansion of the $H$-graph (as defined in Biggs \cite{Biggs_1973,Biggs_AGT}).} 
\label{Fig:Basic Biggs-Smith}
\end{figure}

This paper is structured as follows.
In Section~\ref{sec:BSG} we define the Biggs-Smith graph and establish several of its properties, including the structure of its level sets and a strong property of its automorphism group. We exhibit several different drawings to emphasize the various aspects of its structure that will be relevant for our arguments.
In Section~\ref{sec:mainpf} we prove a key lemma that implies Theorem~\ref{Thm: Lovasz Conj False} by investigating certain special maximum independent sets in the Biggs-Smith graph.
We conclude with a section about future work and a summary of our computational efforts to find other counterexamples to the Lov\'{a}sz
conjecture.

\section{Definition and Properties of the Biggs-Smith Graph}\label{sec:BSG}

We begin with the precise definition of the Biggs-Smith graph, which from now on we will denote by $\BS$. The vertex set consists of $102=17\cdot 6$ vertices, which we label in 17 sets of 6 as $\{1a,1b,1c,1d,1e,1f\}$ up to $\{17a,17b,17c,17d,17e,17f\}$. Each set $H_i=\{ia,ib,ic,id,ie,if\}$ induces a double star in $\BS$ with centers $ie$ and $if$ of degree three each, and with leaves $ia,ib$ adjacent to $ie$ and $ic,id$ adjacent to $if$ (as shown in Figure~\ref{fig:I1}).
The 17 vertices $\{1a,\ldots,17a\}$ induce a 17-cycle, called the \emph{$a$-cycle}, in their natural order, i.e. $ia$ is adjacent to $(i+1)a$ for each $i$ and $17a$ is adjacent to $1a$.  The 17 vertices $\{1b,\ldots,17b\}$ induce the $b$-cycle, a 17-cycle with \emph{step offset} 4, which means $ib$ is adjacent to $(i+4)b$ for each $i$ (reducing the sum $i+4$ modulo 17). Similarly, the $c$-cycle is induced by the 17 vertices $\{1c,\ldots,17c\}$ with step offset 2, and the $d$-cycle by $\{1d,\ldots,17d\}$ with step offset 8 (see Figure~\ref{Fig:Basic Biggs-Smith}). Note then that the resulting graph $\BS$ is cubic (i.e. every vertex has degree 3).
As such $\BS$ is sometimes described as an order-$17$ graph expansion of the $H$-graph (the double star induced by each set $H_i=\{ia,ib,ic,id,ie,if\}$) with step offsets $1, 2, 4$, and $8$. This means taking 17 disjoint copies of $H$ and joining the leaves into four 17-cycles with the prescribed offsets (for more on this see \cite{Biggs_1973}). The drawing on the right in Figure~\ref{Fig:Basic Biggs-Smith} shows the graph with this representation.

\begin{figure}[ht!]
\begin{center}
    \includegraphics[width=0.5\textwidth]{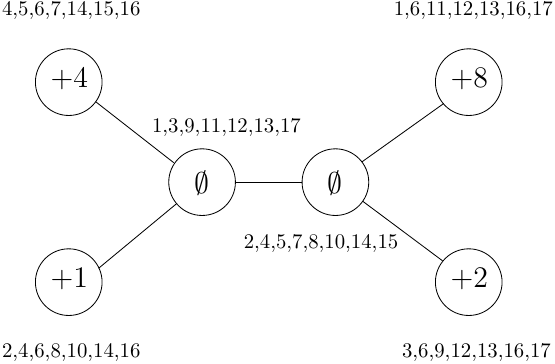}
\end{center}
\caption{A maximum independent set we call $I_1$ in $\BS$ is given.
Note that $I_1$ has size 43.}
\label{fig:double}
\end{figure}

Our next aim in this section is to collect some of the known properties of $\BS$ that will be important for our arguments. 
Notably, the automorphism group $\BS$ is isomorphic to $PSL(2,17)$ per \cite{Biggs_1973}.
We will not require this property specifically, 
however it underlies many of the structural features of $\BS$ we consider.
As mentioned in the introduction, Biggs and Smith~\cite{biggs1971trivalent} discovered $\BS$ in their study of distance-transitive cubic graphs, a property which we now describe.

Let $G = (V,E)$ be a graph.
An \emph{automorphism} of $G$ is a bijection $\phi:V \rightarrow V$ satisfying $\{u,v\} \in E$ if and only if $\{\phi(u),\phi(v)\}\in E$.
A \emph{geodesic} from a vertex $u$ to a vertex $v$ is a shortest path from $u$ to $v$. When we wish to emphasize the endpoints, we write $(u,v)$-\emph{geodesic}. The \emph{distance} $\dist(u,v)$ between $u$ and $v$ in $G$ is the length of a $(u,v)$-geodesic. We write $D_i(u)=\{x:\dist(u,x)=i\}$.
A graph $G$ is \emph{distance-transitive} if for all pairs $u,v \in V$ and $x,y \in V$ such that $\dist(u,v) = \dist(x,y)$, there exists an automorphism $\phi$ so that $\phi(u) = x$ and $\phi(v) = y$.

To say that a graph $G$ with diameter $k$ has \emph{intersection array} $(b_0,\ldots,b_{k-1};c_1,\ldots,c_k)$ means that for every vertex $v$ and each $i$ with $0\leq i\leq k$, every vertex in the set $D_i(v)$ has exactly $b_i$ neighbours in $D_{i+1}(u)$ and exactly $c_i$ neighbours in $D_{i-1}(u)$. (A graph is \emph{distance-regular} if it has a well-defined intersection array, a property that is immediately implied by distance-transitivity, see, e.g. \cite{brouwer1989distance}.)

\begin{theorem}\label{thm:BSprops} The Biggs-Smith graph $\BS$ has the following properties.
  \begin{enumerate}
  \item The graph $\BS$ is distance-transitive~\cite{biggs1971trivalent}.
  \item Its independence number $\alpha(\BS)$ is 43 (see e.g.~\cite{conder2025edge}).
  \item It has diameter 7 (see~\cite{biggs1971trivalent}), girth 9 (see~\cite{conder2025edge}) and intersection array \newline $(3,2,2,2,1,1,1;1,1,1,1,1,1,3)$ (see~\cite{brouwer1989distance}).
  \item It has a Hamilton cycle (see \cite{Biggs_1973}).
  \end{enumerate}
\end{theorem}

All of these statements are established in the literature as indicated. An example of a maximum independent set in $\BS$ is shown in 
Figure~\ref{fig:double}, and a different drawing showing the same set is given in Figure~\ref{fig:I1}.
See Figure~\ref{fig:vertex-levels} for a drawing of $\BS$ that exhibits the properties in (3).
To the authors' knowledge, (2) has only been shown using computer assistance.
For completeness we provide a proof that $\alpha(\BS) = 43$, see the appendix, that uses a computer only for a very simple task.

Since $\BS$ is cubic, Theorem~\ref{thm:BSprops}(4) immediately implies that $\BS$ is 3-edge-colourable, and hence its line hypergraph ${\mathcal L}(\BS)$ is tripartite. Therefore, to prove Theorem~\ref{Thm: Lovasz Conj False} it suffices to prove the following statement.

\begin{lemma}\label{Lemma: Every Distance Can be Avoided}
   For every $\{u,v\},\{x,y\}\in E(\BS)$, there exists a maximum independent set $I\subseteq V(\BS)$ such that
   $$
   I\cap \{u,v,x,y\} = \emptyset.
   $$
\end{lemma}

Parts (1)--(3) of Theorem~\ref{thm:BSprops} will be key to our proof of Lemma~\ref{Lemma: Every Distance Can be Avoided}.

For each vertex $v$ of $\BS$, the set of all \emph{level sets} $D_i(v) = \{u \in V: \dist(u,v) = i\}$ for $v$ forms a partition of $V(\BS)$.
By Theorem~\ref{thm:BSprops}, the choice of the vertex $v$ is not important, and
we write $D_i$ rather than $D_i(v)$ in Figure~\ref{fig:vertex-levels}, which is a drawing of $\BS$ that highlights its level sets.

\begin{figure}[h!]
\begin{center}
\includegraphics[scale = 0.35]{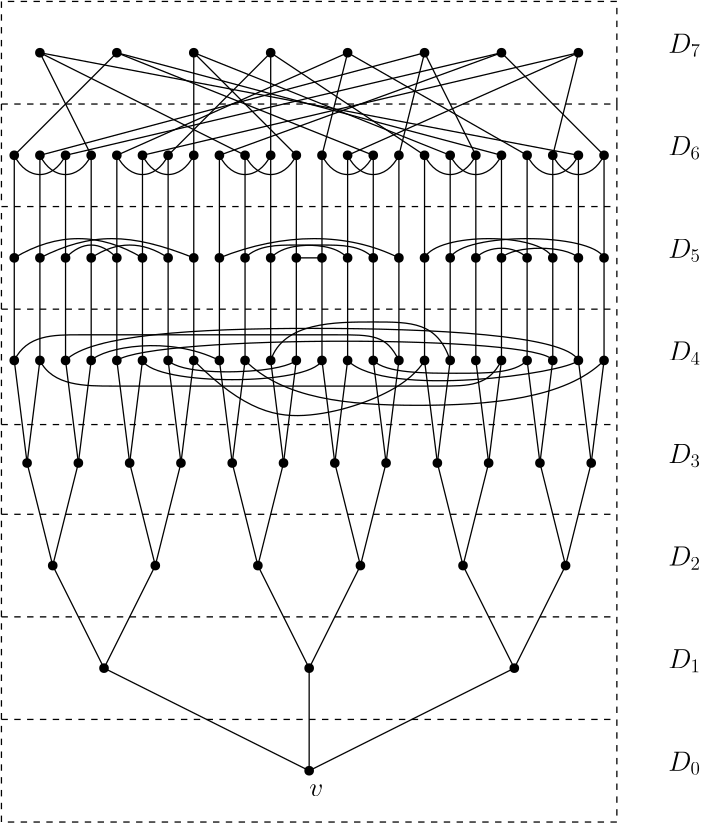}
\end{center}
\caption{A drawing of the Biggs-Smith graph showing the vertex level sets.}
\label{fig:vertex-levels}
\end{figure}

Our upcoming arguments will use some specific consequences of the fact that $\BS$ has intersection array $(b_0,\ldots,b_6;c_1,\ldots,c_7)=(3,2,2,2,1,1,1;1,1,1,1,1,1,3)$ (see also Figure~\ref{fig:vertex-levels}). We state these explicitly in the following lemma for easy reference.

\begin{lemma}\label{Lemma: Vertex Level-sets}
    Let $v \in V(\BS)$ and for $0 \leq i \leq 7$, let $D_i = \{u \in V(\BS): \dist(u,v) = i\}$.
    The following statements hold.
    \begin{enumerate}
      \item For each $i < 7$, each $u \in D_i$ has a unique neighbour in $D_{i-1}$.
        \item For each $i < 7$ and each $u \in D_i$ there exists a unique $(u,v)$-geodesic.
        \item $D_i$ is an independent set for each $i \leq 3$, and also for $i=7$.
        \item $D_i$ induces a matching for each $i$, where $4 \leq i \leq 6$.
        \item No two vertices in $D_7$ have a common neighbour.
        \item For $5\leq i\leq 7$ and for each vertex $z\in D_i$ the following holds. Let $P_z$ be a $(v,z)$-geodesic, let $(x,z)$ be the last edge of $P_z$, and let $y$ be the unique neighbour of $x$ that is not in $P_z$. Then $x,y\in D_{i-1}$.
        \item For any 9-cycle $C$ containing $v$ and any vertex $z$ of $C$ at distance 4 from $v$, the unique neighbour of $z$ that is not on $C$ is in $D_5(v)$.
    \end{enumerate}
\end{lemma}

\begin{proof}
Part (1) holds because $c_i=1$ for each $i<7$, and (2) follows immediately since $c_j=1$ for each $1\leq j\leq i\leq 6$. Part (3) holds because $b_i+c_i=3$ for these values of $i$ (taking $b_7=0$) and so all neighbours of $u$ are in levels different from $D_i$. For $4\leq i\leq 6$ note $b_i+c_i=2$ and hence each vertex in $D_i$ has exactly one neighbour in $D_i$, verifying (4). Part (5) follows from (3) and the fact that $b_6=1$. To verify (6), it is clear from the definition that $x\in D_{i-1}$. Then since $4\leq i-1\leq 6$ we know from (4) that $x$ has exactly one neighbour in $D_{i-1}$, which is therefore $y$. Finally for (7), note that (2) and (4) imply that $z$ has a unique neighbour in $C\cap D_4(v)$. Hence, its unique neighbour not in $C$ is in $D_5(v)$ by (1) and (4).
\end{proof}

We will apply Lemma~\ref{Lemma: Vertex Level-sets} to prove a stronger
property of the automorphism group of the Biggs-Smith graph, which concerns pairs of edges or edge-vertex pairs. We say that a pair of edges $(a_0,a_1)$, $(b_0,b_1)$ is \emph{distance-equivalent} to another pair $(r_0,r_1)$, $(s_0,s_1)$ of edges if for each $(i,j)\in\{0,1\}\times\{0,1\}$ we have $\dist(a_i,b_j)=\dist(r_i,s_j)$. Similarly we say that the edge-vertex pair $(a_0,a_1)$, $b$ is distance-equivalent to another such pair $(r_0,r_1)$, $s$ if for each $(i,j)\in\{0,1\}$ we have $dist(a_i,b)=\dist(r_i,s)$. Here we use the ordered pair notation for edges to emphasise that the definition depends on the ordering of the vertices within that edge.

\begin{lemma}\label{Lemma: BS Edge-Pair Transitive} Let $\BS$ denote the Biggs-Smith graph.
  \begin{enumerate}
  \item    Suppose the edge-vertex pair $(u,v), x$ in $\BS$ is distance-equivalent to the pair $(u',v'),x'$ in $\BS$.
  Then there exists an automorphism $\phi: V(\BS) \rightarrow V(\BS)$ such that $\phi(u) = u'$, $\phi(v) = v'$, and $\phi(x) = x'$.
  \item    Suppose the edge pair $(u,v),(x,y) \in E(\BS)$ is distance-equivalent to the edge pair $(u',v'),(x',y') \in E(\BS)$.
  Then there exists an automorphism $\phi: V(\BS) \rightarrow V(\BS)$ such that $\phi(u) = u'$, $\phi(v) = v'$, $\phi(x) = x'$, and $\phi(y) = y'$.
\end{enumerate}
  \end{lemma}

\begin{proof} For (1), let distance-equivalent edge-vertex pairs be given as in the statement. We may assume
without loss of generality that $\dist(x,v) \leq \dist(x,u)$, and hence $\dist(x',v') \leq \dist(x',u')$.

    \medskip
\noindent\underline{Case 1:} $\dist(x,u)<7$.
\medskip

    By Theorem~\ref{thm:BSprops}(1) (distance-transitivity of $\BS$), we may choose an automorphism $\psi$ of $\BS$ such that $\psi(u) = u'$ and $\psi(x) = x'$.
    We will show that $\psi$ is a suitable choice, i.e. that $\psi(v)=v'$.

    Since $\psi$ is an automorphism we know that $\psi(v)$ is adjacent to $\psi(u)=u'$, and that
    $$\dist(\psi(v),x')=\dist(\psi(v),\psi(x))=\dist(v,x)=\dist(v',x'),$$
    where in the last equality we use the distance-equivalent assumption. Hence $\psi(v)$ and $v'$ are in the same level set $D_i(x')$ of $x'$, and they are both adjacent to $u'$. Since $\dist(x',v') \leq \dist(x',u')$ and $u'$ is a neighbour of $v'$ we know that either $u'\in D_i(x')$ or $u'\in D_{i+1}(x')$. Hence if $\psi(v)\not= v'$ then $u'$ has two distinct neighbours in its own level set or in the one below, neither of which is possible when $\dist(x',u')<7$ by Lemma~\ref{Lemma: Vertex Level-sets}(1,3,4). We therefore conclude that $\psi(v)=v'$.

              \medskip
\noindent\underline{Case 2:} $\dist(u,x)=7$.
\medskip

This time, using Theorem~\ref{thm:BSprops}(1) (distance-transitivity of $\BS$) we choose an automorphism $\psi$ of $\BS$ such that $\psi(v) = v'$ and $\psi(x) = x'$, and our aim is to show that $\psi(u) = u'$.

As before, since $\psi$ is an automorphism we know $\psi(u)$ is adjacent to $\psi(v)=v'$, and 
    $$\dist(\psi(u),x')=\dist(\psi(u),\psi(x))=\dist(u,x)=\dist(u',x')=7.$$
Thus $\psi(u)$ and $u'$ are both in level set $D_7(x')$, and both are adjacent to $v'$, which implies $v'\in D_6(x')$. But by Lemma~\ref{Lemma: Vertex Level-sets}(5), no two vertices in $D_7(x')$ have a common neighbour. Hence $\psi(u)=u'$ as required.

\medskip

Now we consider (2), and proceed with a proof similar to that of Part (1). Let distance-equivalent edge pairs be given as in the statement. We assume
without loss of generality that $\dist(v,y) \leq \dist(v,x)$, and hence $\dist(v',y') \leq \dist(v',x')$.

    \medskip
\noindent\underline{Case 1:} $\dist(v,x)<7$.
\medskip

    By Part (1)  we may choose an automorphism $\psi$ of $\BS$ such that $\psi(u) = u'$, $\psi(v)=v'$  and $\psi(x) = x'$. We will show that $\psi(y)=y'$.

  Since $\psi$ is an automorphism we know that $\psi(y)$ is adjacent to $\psi(x)=x'$, and that
    $$\dist(\psi(y),v')=\dist(\psi(y),\psi(v))=\dist(y,v)=\dist(y',v').$$
Therefore $\psi(y)$ and $y'$ are in the same level set of $v'$, and they are both adjacent to $x'$. By our assumption $\dist(v',y') \leq \dist(v',x')$ we know that $x'$ is either in this same level set or in the one above.  Hence if $\psi(y)\not= y'$ we again find that $x'$ has two distinct neighbours in its own level set or in the one below, neither of which is possible when $\dist(x',u')<7$ by Lemma~\ref{Lemma: Vertex Level-sets}(1,3,4). We therefore conclude that $\psi(y)=y'$.

           \medskip
\noindent\underline{Case 2:} $\dist(v,x)=7$.
\medskip

By Part (1) we may choose an automorphism $\psi$ of $\BS$ such that $\psi(u)=u'$, $\psi(v) = v'$ and $\psi(y) = y'$, and our aim is to show that $\psi(x) = x'$.

Again since $\psi$ is an automorphism we know that $\psi(x)$ is adjacent to $\psi(y)=y'$, and that
    $$\dist(\psi(x),v')=\dist(\psi(x),\psi(v))=\dist(x,v)=\dist(x',v')=7.$$
Thus $\psi(x)$ and $x'$ are both in $D_7(v')$, and both are adjacent to $y'$. But once more by Lemma~\ref{Lemma: Vertex Level-sets}(5), no two vertices in Level 7 have a common neighbour. Hence $\psi(x)=x'$, thus completing the proof.
\end{proof}

The last two results in this section are technical statements that will be useful in the proof of Lemma~\ref{Lemma: Every Distance Can be Avoided}. Their proofs are given in the appendix. For Lemma~\ref{lem:acycle} it is helpful to look at
Figures~\ref{Fig:Basic Biggs-Smith} and~\ref{fig:I1}.

\begin{lemma}\label{lem:acycle}
  Every path of length at most 7 on the $a$-cycle is a geodesic.
\end{lemma}

For Part (2) of the following corollary, the \emph{4-displaced path} $P$ from $x\in D_6(u)$ to $u$ is the path that follows the $(x,u)$-geodesic to $D_4(u)$, then takes the unique edge inside $D_4(u)$ to the resulting vertex $z$, and then follows the  $(z,u)$-geodesic to $u$. Note that $P$ has length $7$, and it is uniquely defined by Lemma~\ref{Lemma:  Vertex Level-sets}(2,4). Figure~\ref{fig:vertex-levels} is a helpful illustration for this corollary.

\begin{corollary}\label{cor:mergingpaths}
The following statements hold for every vertex $u$ in $\BS$.
\begin{enumerate}
\item For each $4\leq i\leq 6$, and for each edge  $(x,y)$ inside a level set $D_i(u)$ of $u$, the $(x,u)$-geodesic and the $(y,u)$-geodesic first meet at a vertex  of $D_{i-4}(u)$.
\item Let $(x,y)$ be an edge inside level set $D_6(u)$. Then the 4-displaced paths from $x$ to $u$ and from $y$ to $u$ meet only at $u$.
\end{enumerate}
\end{corollary}

\section{Proof of Lemma~\ref{Lemma: Every Distance Can be Avoided}}\label{sec:mainpf}

Our goal in this section is to prove Lemma~\ref{Lemma: Every Distance Can be Avoided}, that is, we show that for each edge pair $(u,v),(x,y)$ in $\BS$ there exists a maximum independent set $I$ in $\BS$ that avoids $u,v,x,y$. Our proof will take advantage of the very symmetric properties of $\BS$, which will reduce the number of edge pairs we need to check to a manageable number.

We will make use of a particular maximum independent set in the Biggs-Smith graph, that we call $I_1$.
See Figure~\ref{fig:double} and Figure~\ref{fig:I1}. Note that $|I_1|=43=\alpha(\BS)$.

\begin{figure}[ht!]
\begin{center}
\scalebox{0.75}{
    \includegraphics[width=0.15\textwidth]{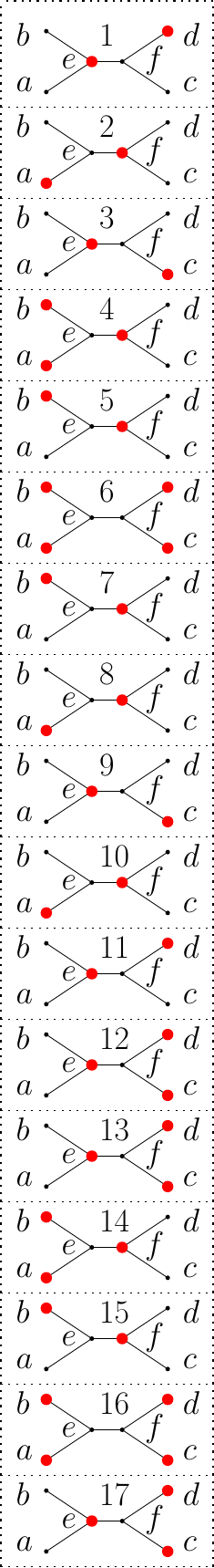}
}
\end{center}
\caption{The vertices of the Biggs-Smith graph.
Edges of the four 17-cycles are not shown. The vertices of the maximum independent set we call $I_1$ are indicated in red.}
\label{fig:I1}
\end{figure}

We introduce the following shorthand notation to denote minor modifications to our independent set $I_1$. For $A\subset I_1$, $C\subset I_1\setminus A$ we write $I_1(A;C)$ for the set
$(I_1\setminus A) \cup C$. 
Often we will drop the set brackets when using this notation,
for example, we write $I_1(x;y)$ to mean $I_1(\{x\};\{y\})$.

We are now ready to prove Lemma~\ref{Lemma: Every Distance Can be Avoided}. It is helpful to consult Figure~\ref{fig:I1} throughout this proof.

\begin{proof}[Proof of Lemma~\ref{Lemma: Every Distance Can be Avoided}]
Let $(u,v),(x,y)\in E(\BS)$ be any pair of distinct edges in $\BS$. 
Since the image of an independent set under an automorphism is also an independent set,
by Lemma~\ref{Lemma: BS Edge-Pair Transitive} and Theorem~\ref{thm:BSprops}(2) ($\alpha(\BS)=43$) it is sufficient
to exhibit an independent set $I$ of size $43$ and a pair of edges $(u',v'),(x',y')$ that is distance-equivalent to the pair $(u,v),(x,y)\in E(\BS)$, such that $I\cap \{u',v',x',y'\} = \emptyset$. 

We assume without loss of generality that the shortest distance from the edge $(u,v)$ to the edge $(x,y)$ 
is $\dist(v,x)=d$. This naturally breaks the proof into cases based on the other pairwise distances amongst $u,v,x,y$.

\medskip
\noindent\underline{Case 1:} $\dist(u,x) = d+1$.
\medskip

Recall $u,v$ are adjacent. As $\dist(u,x) = d+1$ and $\dist(v,x) = d$, the geodesic from $u$ to $x$ that contains $v$ as an interior vertex is a $(u,x)$-geodesic.
We further divide this case into smaller cases, depending on the distance $\dist(u,y)$.

\medskip
\noindent\underline{Case 1.a:} $\dist(u,y) = d+2$.
\medskip

By Theorem~\ref{thm:BSprops}(3) (the diameter of $\BS$ is $7$) we find
$\dist(u,y) = d+2\leq 7$, which implies $d\leq 5$. 
Furthermore, we note that $\dist(v,y)=d+1$. This holds since $\dist(v,x) = d$ and $x,y$ are adjacent, implying $\dist(v,y)\leq d+1$, while $\dist(v,y) \leq d$ would imply $\dist(u,y) \leq d+1$, a contradiction.

For each $d\leq 5$ we provide an independent set $I$ of size $43$ and a pair of edges $(u',v'),(x',y')$ with $\{u',v',x',y,\}\cap I=\emptyset$ and $\dist(v',x') = d$, $\dist(v',y') = d+1$, $\dist(u',x')=d+1$, and $\dist(u',y')=d+2$.
For each possible value of $d$ we provide a $(u',y')$-geodesic that contains $v'$ and $x'$ as interior vertices. For $d\leq 2$ the given path is a geodesic by Theorem~\ref{thm:BSprops}(3) ($\BS$ has girth 9), and for $d\geq 3$ the given path is a geodesic by Lemma~\ref{lem:acycle}. 
\begin{itemize}
    \item $d=0$: edges $(11a,12a),(12a,13a)$, set $I_1$, path $11a,12a,13a$.
    \item $d=1$: edges $(8b,12b),(12e,12f)$, set $I_1(12e;12a)$, path $8b,12b,12e,12f$.
    \item $d=2$: edges $(2c,4c),(8c,10c)$, set $I_1$, path $2c,4c,6c,8c,10c$.
    \item $d=3$: edges $(16a,15a),(12a,11a)$, set $I_1(4b,17e,16a;8b,17b,17a)$, path $16a,15a,14a,\allowbreak 13a,12a,11a$.
    \item $d=4$: edges $(1a,17a),(13a,12a)$, set $I_1$, path $1a,17a,16a,15a,14a,13a,12a$.
    \item $d=5$: edges $(1a,17a),(12a,11a)$, set $I_1$, path $1a,17a,16a,15a,14a,13a,12a,11a$.
\end{itemize}

\medskip
\noindent\underline{Case 1.b:} $\dist(u,y) = d+1$.
\medskip

Then, $x,y \in D_{d+1}(u)$. As $x,y$ are adjacent this implies that $D_{d+1}(u)$ is not an independent set.
Hence by Lemma~\ref{Lemma: Vertex Level-sets}(3) we know $4 \leq d+1 \leq 6$, which implies $3 \leq d \leq 5$.
As $d+1\leq 6$, by Lemma~\ref{Lemma: Vertex Level-sets}(2) there is a unique $(u,x)$-geodesic, and
recall that since we are in Case 1 we know that it contains $v$. 

To determine whether $\dist(v,y)=d$ or $d+1$, we note that  $\dist(v,y)=d$ if and only if the $(x,u)$-geodesic and the $(y,u)$-geodesic meet at $v$. By Corollary~\ref{cor:mergingpaths}(1), whether this happens or not is entirely determined by the value of $d$. Indeed (see also Figure~\ref{fig:vertex-levels}) the corollary implies that if $d=3$ then $\dist(v,y) = d+1 =4$, and if $d = 4,5$, then $\dist(v,y) = d$.

For each $3\leq d\leq 5$ we provide an independent set $I$ of size $43$ and a pair of edges $(u',v'),(x',y')$ with $\{u',v',x',y,\}\cap I=\emptyset$ and $\dist(v',x') = d$, $\dist(v',y')=d$ or $d+1$ as determined by $d$, $\dist(u',x')=d+1$, and $\dist(u',y')=d+1$.
In each case we provide a geodesic $P_z$ of length $d+2=\dist(u',x')+1$ from $u'$ to a neighbour $z\not=y'$ of $x'$, that contains $v'$ and $x'$ as interior vertices. This will then certify that $\dist(u',y')=\dist(u',x')=d+1$ by Lemma~\ref{Lemma: Vertex Level-sets}(6).
The given paths are all geodesics by Lemma~\ref{lem:acycle}.

\begin{itemize}
    \item $d=3$: edges $(1a,17a),(14a,14e)$, set $I_1(13e,14a;13a,13b)$, and path \newline $1a,17a,16a,15a,14a,13a$. 
    \item $d=4$: edges $(1a,17a),(13a,13e)$, set $I_1(13e;13b)$, and path \newline $1a,17a,16a,15a,14a,13a,12a$.
  \item $d$=5: edges $(13a,12a),(7a,7e)$, set $I_1$, and path \newline $13a,12a,11a,10a,9a,8a,7a,6a$.
\end{itemize}

\medskip
\noindent\underline{Case 1.c:} $\dist(u,y) = d$.
\medskip

Note that if $\dist(v,y)=d+1$ then by interchanging the roles of $u$ and $v$, and also the roles of $x$ and $y$, we return to Case 1.b. Hence we may assume that $\dist(v,y)=d$.

Since $d$ was chosen as the smallest pairwise distance between $\{u,v\}$ and $\{x,y\}$, we know that $\dist(v,y),\dist(u,y)\geq d$.
Hence $y$ is not an internal vertex on a $(v,x)$-geodesic. This tells us that there are two distinct $(u,x)$-geodesics (both of length $d+1$), one starting with the edge $(u,v)$ and the other ending with the edge $(y,x)$. Hence by Lemma~\ref{Lemma: Vertex Level-sets}(1) we find $d+1=7$ and so $d=6$.

We provide an independent set $I$ of size $43$ and a pair of edges $(u',v'),(x',y')$ with $\{u',v',x',y'\}\cap I=\emptyset$ and $\dist(v',x') = 6$, $\dist(v',y') = 6$, $\dist(u',x')=7$, and $\dist(u',y')= 6$.
We also provide two internally-disjoint $(u',x')$-geodesics, one containing $v'$, the other containing $y'$, as internal vertices.
\begin{itemize}
    \item edges $(17a,1a),(7a,7e)$, set $I_1$, \newline 
    the $(u',x')$-geodesic containing $v'$ is $17a,1a,2a,3a,4a,5a,6a,7a$, and this is a geodesic by Lemma~\ref{lem:acycle},\newline
    the $(u',x')$-geodesic containing $y'$ is $17a,16a,16e,16b,3b,7b,7e,7a$, which is also a geodesic since it has the same length.
\end{itemize}

\noindent This concludes Case 1. \hfill $\diamond$

\medskip
\noindent\underline{Case 2:} $\dist(u,x) = d$.
\medskip

If $\dist(v,y)=d+1$ or $\dist(u,y)=d+1$ then (as in Case 1.c) we may interchange the roles of $u$ and $v$, and/or the roles of $x$ and $y$, and reduce to a previous case. Hence we may assume $\dist(v,y)=d$ and $\dist(u,y)=d$, i.e. all four distances are equal to $d$. For this proof it helps to consult Figure~\ref{fig:vertex-levels}.

First we claim that $d=5$ or $d=6$. Note that the edge $(u,v)$ lies inside $ D_{d}(x)$, and hence $4\leq d\leq 6$ by Lemma~\ref{Lemma: Vertex Level-sets}(3). Since $\dist(y,u)=\dist(y,v)=d$ we know $(y,x)$ is not the last edge of a $(v,x)$-geodesic or a $(u,x)$-geodesic. Hence for a path $P$ of length $d$ to reach $y\in D_1(x)$ from $u\in D_d(x)$, $P$ must contain exactly one vertex from each level $D_d(x),\ldots,D_1(x)$ except one $D_i(x)$, in which it contains exactly two vertices, that are adjacent. Hence $i\geq 4$ by Lemma~\ref{Lemma: Vertex Level-sets}(3). But $i=d$ is not possible because $(u,v)$ is the unique edge in $D_d(x)$ incident to $u$ by Lemma~\ref{Lemma: Vertex Level-sets}(4), and $\dist(y,v)=d$. Hence $d>i\geq 4$. Hence we conclude $d=5$ or $d=6$ as claimed.

Next we eliminate the possibility that $d=6$. In this case, first suppose the path $P$ contains an edge $(z,w)$ inside $D_5(x)$, where the first edge of $P$ is $(u,z)$. By Corollary~\ref{cor:mergingpaths}(1) we know that the $(z,x)$-geodesic and the $(w,x)$-geodesic contain the same neighbour $t$ of $x$. Since the $(u,x)$ geodesic is $(u,z)$ followed by the $(z,x)$-geodesic, we know $t\not=y$. But $P$ must coincide with the $(w,x)$-geodesic on levels $D_4(x),\ldots,D_1(x)$, so $P$ ends with $t\not=y$. This contradiction shows that if $P$ exists then $P=Q-x$, where $Q$  is the 4-displaced path from $u$ to $x$. In addition there exists a path $P'$ of length 6 from $v$ to $y$, which by the same arguments must be $Q'-x$ where $Q'$ is the 4-displaced path from $v$ to $x$. But Corollary~\ref{cor:mergingpaths}(2) tells us that $Q$ and $Q'$ meet only at $x$, contradicting the requirement that $P$ and $P'$ end at $y$. Hence the $d=6$ case does not occur.

    Thus we only need to verify the case $d=5$. Here we take 
 edges $(2b,2e),(8c,10c)$ and set $I_1$. To certify that all four distances are 5 we apply Lemma~\ref{Lemma: Vertex Level-sets}(7).

 $\dist(2e,8c)=5$ via 9-cycle $C_1=2e,2f,2c,4c,6c,6f,6e,6b,2b$ and edge $(6c,8c)$.

 $\dist(2e,10c)=5$ via 9-cycle $C_2=2e,2f,2d,10d,10f,10e,10b,6b,2b$ and edge $(10f,10c)$.

 $\dist(2b,8c)=5$ via 9-cycle $C_1$ (in the opposite direction) $2b,6b,6e,6f,6c,4c,2c,2f,2e$ and edge $(6c,8c)$.

 $\dist(2b,10c)=5$ via 9-cycle $C_2$ (in the opposite direction) $2b,6b,10b,10e,10f,10d,2d,2f,2e$ and edge $(10f,10c)$.

\noindent This concludes Case 2. \hfill $\diamond$
\medskip

Therefore, we have shown for every possible set of pairwise distances between edges $(u,v)$ and $(x,y)$ in $\BS$, there exists an independent set $I$ of size $43$ and a pair of edges $(u',v'),(x',y')$ 
that is distance-equivalent to $(u,v),(x,y)$
such that $I\cap \{u',v',x',y'\} = \emptyset$. 
Using Lemma~\ref{Lemma: BS Edge-Pair Transitive}(2) this completes the proof of Lemma~\ref{Lemma: Every Distance Can be Avoided}.
\end{proof}

\section{Future Work}\label{sec:future}

After proving Lemma~\ref{Lemma: Every Distance Can be Avoided}, it is natural to ask are there other cubic graphs with the same property as $\BS$?
Is there a smaller example?
In order to provide a complete picture to the reader, as well as set the stage for future work, we have taken the first steps to answering these questions.

With computer assistance we have verified that the Biggs-Smith graph is the unique cubic edge-transitive graph on at most $166$ vertices that satisfies the conclusion of Lemma~\ref{Lemma: Every Distance Can be Avoided}.
On $168$ vertices, there is another edge-transitive example.
This $168$-vertex graph is labeled as graph F168D in the Encylopedia of Graphs page on the the Foster census.
Hence, the line hypergraph of the Biggs-Smith graph is not the only counterexample to the Lov\'{a}sz conjecture.
We have further verified that F168D is the only cubic edge-transitive graph on at most $300$ vertices with this property other than the Biggs-Smith graph.

In order to check this, we used
the full list of cubic edge-transitive graphs with at most $300$ vertices produced by Conder, and Dobcs{\'a}nyi \cite{conder2002trivalent} and Conder, Malni{\v{c}}, Maru{\v{s}}i{\v{c}}, and Poto{\v{c}}nik \cite{conder2006census}.
These lists have recently been extended to higher orders and 
conveniently summarized by Conder and Poto{\v{c}}nik in \cite{conder2025edge}. 
Our calculations were run on a personal computer and only required several hours.
Additionally, we have also verified, using nauty \cite{mckay2014practical}, that no cubic graph on at most $20$ vertices satisfies the conclusion of Lemma~\ref{Lemma: Every Distance Can be Avoided}.

This leads to the following natural questions.

\begin{question}
    Do there exist infinitely many $3$-edge colourable cubic graphs $G$ such that for all edges $(u,v),(x,y) \in E(G)$, 
    $$
    \alpha(G) = \alpha(G - \{u,v,x,y\})?
    $$
\end{question}

\begin{question}
    Does there exist a $d$-regular $d$-edge colourable graph $G$ where $d\geq 4$ such that for all edges $(u_1,v_1),\dots, (u_{d-1}, v_{d-1}) \in E(G)$, 
    $$
    \alpha(G) = \alpha(G - \{u_1,\dots,u_{d-1}, v_1,\dots,v_{d-1} \})?
    $$
\end{question}

Despite our efforts to find a large number of counterexamples, we have only found two.
That is to say, the Lov\'{a}sz conjecture appears true for most line hypergraphs of cubic graphs.
Given this, does there exist a slightly weaker version of the conjecture that may be true?
We propose the following weakening, which still
 implies Ryser's conjecture.
We have verified Conjecture~\ref{Conjecture: weak version} holds for the line hypergraph of the Biggs-Smith graph.

\begin{conjecture}\label{Conjecture: weak version}
    For every $r$-partite $r$-uniform hypergraph $\mathcal{H}$ there exists some $k$ with $1\leq k\leq r-1$ such that $\mathcal{H}$
    contains $k(r-1)$ vertices whose deletion reduces the matching number by at least $k$.
\end{conjecture}

Significantly, the Biggs-Smith graph and F168D are not Cayley graphs.
Hence, the only known counterexamples to the Lov\'{a}sz conjecture are both line hypergraphs of highly symmetric graphs that are not Cayley graphs.
While it is easy to see why symmetry in a prospective counterexample is helpful (at least in identifying it as a counterexample), it does not seem obvious to the authors why not being a Cayley graph is beneficial.

\begin{question}
    Is it true that
    if $\mathcal{L}$ is the line hypergraph of an $r$-regular $r$-edge colourable Cayley graph, then
    $\mathcal{L}$ contains $r-1$ vertices whose deletion reduces its matching number?
\end{question}

\section*{Acknowledgments}
This research began at the Banff International Research Station workshop on New Perspectives in Colouring and Structure in October 2024. The authors would like to thank Maya Stein, Youri Tamitegama and Jane Tan for insightful initial discussions about this problem.

\bibliographystyle{plain}
\bibliography{bib}


\appendix

\section{Appendix}

In the appendix, we prove some of the properties of the Biggs-Smith graph that have been used in the paper. We start with the proof of 
Lemma~\ref{lem:acycle} that states that every path of length at most 7 on the $a$-cycle is a geodesic.
For the proof it is helpful to consult Figures~\ref{Fig:Basic Biggs-Smith} and~\ref{fig:I1}. Recall also that
the vertex set of $\BS$ is partitioned into sets $H_1,\dots, H_{17}$ where $H_i = \{ia,ib,ic,id,ie,if\}$, see Figure~\ref{fig:I1}.
Each edge of the $a$-cycle joins a vertex  $ia\in H_i$ to the $a$-vertex of $H_{i \pm 1}$.
Similarly, each edge of the $b$, $c$, and $d$-cycles joins a vertex of $H_i$ to a vertex in $H_{i \pm 4}$, $H_{i \pm 2}$, and $H_{i \pm 8}$, respectively (where all indices $i\pm k$ are taken modulo 17).

\begin{proof}[Proof of Lemma~\ref{lem:acycle}]

  By Theorem~\ref{thm:BSprops}(3) ($\BS$ has girth 9) it is immediate that every path in $\BS$ of length at most 4 is a geodesic.
  
  For any pair of vertices $x$ and $y$ on the $a$-cycle, if a path $P$ from $x$ to $y$  contains a vertex of the $b$-cycle, then to reach the $b$-cycle from $x$ and return to $y$, $P$ must also contain at least two vertices $ie$ and $je$ for some $j\not=i$, and at least two vertices of the $b$-cycle (see Figure~\ref{fig:double}). Hence $P$ has length at least 5. Thus any path on the $a$-cycle of length 5 is a geodesic. Similarly, if $P$ contains a vertex of the $c$-cycle or the $d$-cycle, then $P$ must also contain at least four vertices $ie,if$ and $je,jf$ for some $j\not=i$, 
  two vertices $kf$ and $\ell f$ with $k\not=\ell$ (where $k=i$ and $\ell=j$ unless $P$ also visits yet another cycle), 
  and at least two vertices of the $c$-cycle or $d$-cycle. This implies $P$ has length at least 7, implying that we may disregard any such paths from now on.

  Next we consider a path on the $a$-cycle of length 6 (or 7). By symmetry of $\BS$, we may assume this path is $P=1a,2a,\dots,7a$ (or $P=1a,2a,\dots,8a$). By the conclusions in the previous paragraph, any path $Q\ne P$ of length at most 6 from $1a$ to $7a$ (or $8a$) must contain two subpaths of length 2 from the $a$-cycle to the $b$-cycle and back. These two paths are of the form $ia,ie,ib$ and $jb,je,ja$ (where $i\ne j$).
  In addition, $Q$ can contain at most 2 edges on the $a$-cycle or the $b$-cycle. As previously noted, each of these additional edges crosses between a set $H_k$ and a set 
  $H_{k\pm 1}$ or $H_{k\pm 4}$.
  The path $Q$ begins at $1a\in H_1$, while it ends at $7a\in H_7$ (or $8a\in H_8$).
  Since neither 6 nor 7 can be written as the sum of at most two elements of $\{\pm1,\pm4\}$ (modulo 17), we see that $Q$ cannot start in $H_1$ and reach $H_7$ (or $H_8$) using at most two edges on the $a$-cycle and the $b$-cycle. 
  This shows that $P$ is a geodesic.  
\end{proof}

For the proof of Corollary~\ref{cor:mergingpaths}, the reader should consult Figure~\ref{fig:vertex-levels}.

\begin{proof}[Proof of Corollary~\ref{cor:mergingpaths}]
  For both parts, first note that by Lemma~\ref{Lemma: Vertex Level-sets}(2), all the paths in question are uniquely determined. Also, by Lemma~\ref{Lemma: BS Edge-Pair Transitive}(1), for any other edge $(x',y')$ in the same level set as $(x,y)$, there exists an automorphism $\phi$ of $\BS$ with $\phi(x)=x'$, $\phi(y)=y'$, and $\phi(u)=u$. Hence for (1) we know that $\phi$ maps the $(x,u)$-geodesic to the  $(x',u)$-geodesic, and the $(y,u)$-geodesic to the  $(y',u)$-geodesic. Similarly for (2) we find $\phi$ maps the 4-displaced paths from $x\in D_6(u)$ and from $y\in D_6(u)$ to the 4-displaced paths from $x'\in D_6(u)$ and from $y'\in D_6(u)$, respectively.
  Hence to verify (1) and (2) it suffices to prove each case for one choice of $(x,y)$ and $u$. For both parts we choose $u=17a$, and $(x,y)=(ia,ie)$, where $4\leq i\leq 6$ for (1), and just $i=6$ for (2). By Lemma~\ref{Lemma: Vertex Level-sets}(6), since $17a,1a,2a,3a,4a,5a,6a,7a$ is a geodesic (by Lemma~\ref{lem:acycle}) it follows that $(ia,ie)\in D_i(17a)$.

 For (1)
  when $i=4$, the path $4e,4b,17b,17e,17a$ has length 4 and is thus the unique $(4e,17a)$-geodesic. It meets the $(4a,17a)$-geodesic $4a,3a,2a,1a,17a$ only at $17a$.

  For $i=5$, the path $5e,5b,1b,1e,1a,17a$ has length 5 and is thus the unique $(5e,17a)$-geodesic. It first meets the $(5a,17a)$-geodesic $5a,4a,3a,2a,1a,17a$  at $1a\in D_1(17a)$.

  When $i=6$, the path $6e,6b,2b,2e,2a,1a,17a$ has length 6 and is thus the unique $(6e,17a)$-geodesic. It first meets the $(6a,17a)$-geodesic $6a,5a,4a,3a,2a,1a,17a$  at $2a\in D_2(17a)$. This completes the proof of (1).

  Now we consider (2). Using the geodesics determined in the proof of (1) we see that the 4-displaced path from $6a$ is $6a,5a,4a,4e,4b,17b,17e,17a$. Since $6e,6b,2b,2e,2a,1a,17a$ is a geodesic (as found in the proof of (1)) we know by Lemma~\ref{Lemma: Vertex Level-sets}(6) that the edge $(2b,15b)$ is in $D_4(17a)$. Then since the path $15b,15e,15a,16a,17a$ has length 4, we know that it is the $(15b,17a)$-geodesic.   Hence, the 4-displaced path from $6e$ is $6e,6b,2b,15b,15e,15a,16a,17a$, which meets the 4-displaced path from $6a$ only at $17a$. Thus, (2) is proved.
\end{proof}

We conclude by providing a proof that $\alpha(\BS) = 43$ which uses computer assistance only for a very simple task. The proof shows how tight this result is. The proof will involve automorphisms of $\BS$ that preserve the partition into the six sets $A,B,C,D,E,F$ of the $H$-representation of the graph, where $A=\{1a,\dots, 17a\}$, $B=\{1b,\dots,17b\}$, etc. We will refer to these maps simply as "symmetries". Below we use the notation $[17]=\{1,\ldots,17\}$.

\begin{proposition}\label{prop:symmetries}
    The automorphisms of $\BS$ that preserve the partition of $V(\BS)$ into sets $A,B,C,D,E,F$ of the $H$-representation of $\BS$ correspond to permutations of the sets 
    $H_i = \{ia,ib,ic,id,ie,if\}$, $i\in [17]$. They are in bijective correspondence with permutations of $[17]$ that can be written uniquely as the product of a dihedral permutation $\alpha$ of $[17]$ and a permutation $\phi_k: i\mapsto ki\pmod{17}$, where $k = 1,2,4,8$, and $0 \pmod{17}$ is written as $17$. The product $\alpha\phi_k$ preserves each part of the partition if $k=1$, interchanges $A$ with $B$ (and $C$ with $D$) if $k=4$, and interchanges $E$ with $F$ (and $\{A,B\}$ with $\{C,D\}$) if $k=2$ or $8$.
\end{proposition}

\begin{proof}
    The automorphisms of $\BS$ that preserve the $H$-representation and preserve the partition into sets $A,B,\dots,F$ are obtained by permuting the $H$-sets $H_i$, $i\in[17]$.
    Thus, these automorphisms correspond to permutations of the set [17].
    Clearly, the dihedral permutations of [17] preserve the $a$-, $b$-, $c$-, and $d$-cycles. 
    The permutation $\phi_k$ of [17] acting on the index of the sets $H_i$
    changes the $a$-cycle into a cycle with step $k$ and changes the $b$-cycle into the cycle with step $4k$ (which is the $a$-cycle if $k=4$, the $d$-cycle if $k=2$, and the $c$-cycle if $k=8$).
    The rest of the proof is apparent, and the details are left to the reader. 
\end{proof}

\begin{proposition}
    $\alpha(\BS)=43$.
\end{proposition}

\begin{proof}
    It is easy to verify that the set $I_1$ given in Figures~\ref{fig:double} and ~\ref{fig:I1} is an independent set of order $43$ in $\BS$. Hence, $\alpha(\BS) \geq 43$.

    For a graph $G$ we denote by $\cI(G)$ the set of all independent sets in $G$.


    For $J\in\cI(\BS)$, we let $J_a$ be the set of all $i \in [17]$ such that $ia \in J$ and let $j_a = |J_a|$.
    Define $J_b,J_c,J_d,J_e,J_f$ and $j_b,j_c,j_d,j_e,j_f$ analogously with respect to the sets $B,C,D,E,F$.
    Now, given the $H$-structure of $\BS$, it is clear that
    \begin{enumerate}
        \item[(1)] $J_e\cap J_f = \emptyset$, $J_e \cap (J_a\cup J_b) = \emptyset$, and $J_f \cap (J_c\cup J_d) = \emptyset$.
    \end{enumerate}
    The following inequalities are implied, or trivial in the case of (2):
    \begin{enumerate}
        \item[(2)] for all $x \in \{a,b,c,d\}$, $j_x\leq 8$,
        \item[(3)] $j_e + j_f \leq 17$,
        \item[(4)] $j_0 := j_a + j_b + j_e \leq 17 + |J_a \cap J_b|$, and $j_1 := j_c + j_d + j_f\leq 17 + |J_c \cap J_d|$.
    \end{enumerate}

We write $\BS_0$ for the subgraph of $\BS$ induced by $A\cup B\cup E$, and similarly $\BS_1=\BS[C\cup D\cup F]$. 
We begin our proof by establishing some basic properties of the elements $J\in\cI(\BS_0)$. Note that for such sets $|J|=j_0$. By symmetry these facts will have analogues for $\cI(\BS_1)$.

For $J\in\cI(\BS_0)$ and two sets $Q,Q'\subseteq [17]$ where each is contained in one of $J_a,J_b,J_e$, we will say that $Q$ and $Q'$ are \emph{equivalent}, $Q\sim Q'$, if one can be obtained from the other by applying an automorphism of $\BS$ which maps $E = \{e_1,\dots, e_{17}\}$ to itself. 
By Proposition \ref{prop:symmetries}, these automorphisms correspond to the action of the dihedral group ${\mathcal{D}}_{17}$ (which also preserves $A$ and $B$) together with the permutation $\phi_4: [17]\to [17]$ that corresponds to the mapping $i\mapsto 4i \pmod{17}$ (which exchanges $A$ and $B$) such that $0 \pmod{17}$ is written as $17$.

\medskip
\underline{Claim 1:} If $J\in\cI(\BS_0)$ and $j_b = 8$, then $j_0 \leq 21$.
\medskip

    Suppose that $j_b = 8$. Up to symmetries that fix the $b$-cycle, we may assume that $$J_b = \{1,9,17,8,16,7,15,6\}.$$
    As the $a$-cycle has step length $1$, and $i,j \in J_a$ implies $ia,ja$ are not adjacent on the $a$-cycle, it is trivial to verify that $|J_a \cap J_b| \leq 4$. Hence, $j_0 \leq 17 + |J_a \cap J_b| \leq 21$ as claimed. \hfill $\diamond$

\bigskip

Note for example that due to symmetries, Claim 1 implies that if $j_a\ge8$ then $j_0\le21$. 

In the upcoming arguments we often consider the sets $J_a\cap J_b$. Such sets are independent sets in the circulant graph $C(17,\{1,4\})$ that has vertex set $[17]$ in circular order, with edges $ij$ if and only if $|i-j|\in \{1,4\}$ (the difference taken modulo $17$), i.e. the ``superposition" of the $a$-cycle and the $b$-cycle. When considering $J_a\cap J_b$, it is helpful to keep in mind that the circular distance in $[17]$ from one vertex in $J_a\cap J_b$ to the next is never 1 or 4, and also two consecutive such distances cannot both be 2. We write $j_{ab}=|J_a\cap J_b|$.

\medskip
\underline{Claim 2:} For $J\in\cI(\BS_0)$ we have $j_{ab}\le6$, and hence $j_0\le23$.
\medskip

    Take any interval $P$ of 5 consecutive vertices on the $a$-cycle. It is clear that $|P\cap (J_a\cap J_b)|\le2$. If we count the number of pairs $(P,v)$, where $v\in P\cap (J_a\cap J_b)$, we obtain at most 34 such pairs. But each $v\in J_a\cap J_b$ belongs to 5 such intervals, thus $5|J_a\cap J_b|\le34$. This shows that $j_{ab} < 7$. Finally, (4) implies that $j_0\le23$.
    \hfill $\diamond$

\bigskip


We say that $J\in \cI(\BS_0)$ is \emph{left-maximal} if it has the following properties: 

\begin{itemize}
    \item for each $x\in (A\cup B)\setminus J$, the set $J\cup\{x\}$ is not independent, and 
    \item whenever $i\in J_e$, either $i-1$ or $i+1$ belongs to $J_a$ (otherwise we could replace $ie$ with $ia$) and either $i+4$ or $i-4\pmod{17}$ is in $J_b$ (since otherwise we could replace $ie$ with $ib$).
\end{itemize}
 
    
\medskip
\underline{Claim 3:} Suppose $J\in\cI(\BS_0)$ is left-maximal and $j_{ab} = 6$. Then the following hold.
\begin{itemize}
\item $j_0\leq j_e+13$.
    \item If $j_0=23$, then $j_e=10$, $J_e\sim \{2,4,5,7,9,10,12,\allowbreak 14,16,17\}$ and $J_a$ and $J_b$ are equivalent to $\{1,3,6,8,11,13,15\}$ and $\{1,3,6,8,11,13\}$.
    \item If $j_0=22$, then $J_a$ and $J_b$ are still equivalent to $\{1,3,6,8,11,13,15\}$ and $\{1,3,6,8,11,13\}$, and $J_e$ is as above with one of its elements removed.
\end{itemize} 
\medskip 
It is straightforward to verify that $J_{ab}=J_a\cap J_b$ is equivalent to $\{1,3,6,8,11,13\}$  or $\{1,3,6,9,12,15\}$ (using dihedral symmetries). However, these two sets are equivalent (by composing dihedral symmetries and the permutation $\phi_4$). So, we may assume that we have the former one. 
It is easy to see that one can only add element 15 or (exclusive) 16 to $J_a$ if $J_{ab}=\{1,3,6,8,11,13\}$. Note that these two choices differ by dihedral symmetry. 
The reader can also verify that no elements may be added to $J_b$. Hence $j_a+j_b\leq 13$, and so $j_0=j_a+j_b+j_e\leq 13+j_e$.

If $j_0=23$ then $j_e\geq10$. From the previous paragraph we may assume $J_e\subseteq \{2,4,5,7,9,10,\allowbreak 12,14,15,16,17\}$, so if $j_e=11$ then this is the whole set $J_e$. But then $J$ is not left-maximal since 15 or 16 could be moved to $J_a$. Hence $j_e=10$,  $J_a\sim\{1,3,6,8,11,13,15\}$, and $J_b=J_{ab}$.
This is the only way to obtain $j_0=23$ up to symmetries.

If $j_0=22$, we may reach the same conclusion as above with one of the elements $x$ from $J_e$ removed (see case 6.1 in Table~\ref{table:jab=5}). 
Otherwise, $J_a=J_b \sim \{1,3,6,8,11,13\}$ and $J_e\subseteq \{2,4,5,7,9,10,12,14,15,16,17\}$, in which case again $J$ is not left-maximal, since 15 or 16 could be moved to $J_a$. Hence this 
possibility cannot occur.
        \hfill $\diamond$

    \medskip
    \underline{Claim 4:} Suppose $J\in\cI(\BS_0)$ has size $j_0 = 22$. Then $|J_a\cap J_b| \ge 5$ and $j_e\ge 8$. Moreover, if $J$ is left-maximal and $|J_e|<10$, then $J_e$ is equivalent to one of the cases shown in Table \ref{table:jab=5}.  
    \medskip

\begin{table}[ht]
    \centering
    \begin{tabular}{|c|c|c|c|l|c|}
        \hline
        Case & $J_a\cap J_b$ & $J_a\setminus J_b$ & $J_b\setminus J_a$ & $J_e$ & $j_e$ \\ 
        \hline
        6.1 & 1 3 6 8 11 13 & 15 & -- & $\{2,4,5,7,9,10,12,14,16,17\}\setminus\{x\}$ & 9 \\ 
        \hline
        5.1 & 1 3 6 8 11 & 13 15 & ~9 17 & 2 4 5 7 10 12 14 16 & 8 \\ 
            &  & 13 16 & ~9 17 & 2 4 5 7 10 12 14 15 & 8 \\ 
            &  & 14 16 & ~9 17 & 2 4 5 7 10 12 13 15 & 8 \\ 
            &  & 14 16 & 13 & 2 4 5 7 10 12 9 15 17 & 9 \\ 
        \hline
        5.2 & 1 3 6 9 11 & 13 15 & 4 12 & 2 5 7 10 8 14 16 17 & 8 \\ 
         &  & 13 16 & 4 12 & 2 5 7 10 8 14 15 17 & 8 \\ 
         &  & 14 16 & 4 12 & 2 5 7 10 8 13 15 17 & 8 \\ 
         &  & 13 15 & 8 17 & 2 5 7 10 4 12 14 16 & 8 \\ 
         &  & 13 16 & 8 17 & 2 5 7 10 4 12 14 15 & 8 \\ 
         &  & 14 16 & 8 17 & 2 5 7 10 4 12 13 15 & 8 \\ 
         &  & 13 15 & 12 17 & 2 5 7 10 4 8 14 16 & 8 \\ 
         &  & 13 16 & 12 17 & 2 5 7 10 4 8 14 15 & 8 \\ 
         &  & 14 16 & 12 17 & 2 5 7 10 4 8 13 15 & 8 \\ 
        \hline
        5.3 & 1 3 6 9 12 & 14 16 & 4 11 & 2 5 7 8 10 13 15 17 & 8 \\ 
         &  & 14 16 & 11 17 & 2 5 7 8 10 13 4 15 & 8 \\ 
         &  & 14 16 & 4 15 & 2 5 7 8 10 13 11 17 & 8 \\ 
         &  & 14 16 & 15 17 & 2 5 7 8 10 13 4 11 & 8 \\ 
         &  & 15 & 4 11 & 2 5 7 8 10 13 14 16 17 & 9 \\ 
         &  & 15 & 11 17 & 2 5 7 8 10 13 4 14 16 & 9 \\ 
        \hline
        5.4 & 1 3 6 8 13 & 10 15 & 11 & 2 4 5 7 9 12 14 17 16 & 9 \\ 
         &  & 10 16 & 11 & 2 4 5 7 9 12 14 17 15 & 9 \\ 
         &  & 10 16 & 15 & 2 4 5 7 9 12 14 17 16 & 9 \\ 
         &  & 11 16 & 15 & 2 4 5 7 9 12 14 17 10 & 9 \\ 
        \hline
    \end{tabular}
    \caption{Possible left-maximal cases where $j_0=22$ and $j_e<10$. If $|J_a\cap J_b|=6$, an arbitrary element $x$ is to be removed from the list for $J_e$.}\label{table:jab=5}
\end{table}

By (4) we conclude that $|J_a\cap J_b| \geq 5$, and Claim 1 together with its symmetric version for $j_a$ imply that $j_e\geq 8$. Suppose that $J$ is left-maximal and $j_e\leq 9$. If $|J_a\cap J_b|=6$ then by Claim 3 the situation is as described in the first line of the table (Case 6.1). To prove the remainder of the claim, assume that $|J_a\cap J_b|=5$. We first establish that there are four equivalence classes for a 5-element set $J_a\cap J_b$ as listed in the table under 5.1--5.4. Recall that such a set must be an independent set in the circulant $C(17, \{1,4\})$. For equivalence we use dihedral symmetries and $\phi_4$.

For each case we then determine which elements could be added into $J_a$ and which into $J_b$. The third and the fourth column of the table contain all possibilities of extending $J_a$ and $J_b$ that are left-maximal and do not introduce another element in $J_a\cap J_b$. The fifth column then lists the corresponding sets $J_e$. The last possibility, Case 5.4, needs an additional explanation why we do not have $J_a\setminus J_b=\{11,15\}$. The reason is that in such a case, 10 and 16 would be in $J_e$ and we would have $j_e=10$.
\hfill $\diamond$

\bigskip
As mentioned earlier, the elements in [17] have natural circular distance. More precisely, we write $d(i,j)=k$ if $k\le8$ and $j=i\pm k\pmod{17}$.

    \medskip
    \underline{Claim 5:} Let $J,J' \subseteq [17]$ be such that $J\sim J'$.
    If there exist pairs $i_1,i_2$ and $i_3,i_4$ in $J'$ such that $d(i_1,i_2) = 2$ and $d(i_1,i_2) = 8$,
    then there exist pairs $k_1,k_2$ and $k_3,k_4$ in $J$ such that $d(k_1,k_2) = 2$ and $d(k_3,k_4) = 8$.
    \medskip

    Let $J$ and $J'$ be as described and let $i_1,i_2,i_3,i_4 \in J'$ be such that $d(i_1,i_2) = 2$ and $d(i_1,i_2)~=~8$.
    Trivially,  dihedral symmetries preserve circular distances, so if $J = \beta(J')$, we may let $k_x = \beta(i_x)$ for all $x \in [4]$.
    Next, observe that 
    any composition, $f \circ \beta(J')$ or $\beta\circ f(J')$, of a bijection $f:[17] \rightarrow [17]$ and a dihedral symmetry $\beta$ applied to $J'$ has the same circular distances as $f(J')$.
    Hence it is sufficient to show that applying $\phi_4$ or $\phi_4^{-1}$ 
    to $J'$ preserves the desired property. 
    
    Consider $\phi_4(J')$. 
    Since $d(i_1,i_2) = 2$ without loss of generality we may write $i_1 = i_2+2~\pmod{17}$.
    Then, 
    \begin{align}
        \phi_4(i_1) = 4i_1 = 4i_2 + 8 \pmod{17}
    \end{align}
    implying that $d(\phi_4(i_1), \phi_4(i_2)) = 8$.
    Similarly, $d(\phi_4(i_3), \phi_4(i_4)) = 2$.
    So we can take $k_1~=~\phi_4(i_3), k_2 = \phi_4(i_4)$, and we may take $k_3 = \phi_4(i_1), k_4 = \phi_4(i_2)$.
    The case $\phi_4^{-1}(J')$ follows by the same argument.
    \hfill $\diamond$
     


    \medskip
    \underline{Claim 6:} Suppose that $J\in\cI(\BS_0)$ is left-maximal, $j_e\le7$ and $j_0 \ge21$. Then $j_0=21$ and $J_e$ contains two pairs $i_1,i_2$ and $i_3,i_4$ such that $d(i_1,i_2)=2$ and $d(i_3,i_4)=8$. 
    \medskip

    If on the contrary $j_0=j_a+j_b+j_e\geq 22$, then $j_a+j_b\geq 15$. But then (without loss of generality) $j_b=8$, which by Claim 2 implies the contradiction $j_0\leq 21$. Hence $j_0=21$. 
    
    By (4) we must have $|J_a\cap J_b|\ge4$. If $|J_a\cap J_b|=6$, then Claim 3 implies that $j_0\le 13+j_e\leq 20$, a contradiction. Suppose now that $|J_a\cap J_b|=5$. Since $j_0=21$, there is precisely one element $i\in [17]$ that does not belong to $J=J_a\cup J_b\cup J_e$. Since $J$ is left-maximal we know that neither $J\cup\{ia\}$ nor $J\cup\{ib\}$ is independent. Therefore  $J'=J\cup\{ie\}$ is independent, and moreover it is left-maximal. Since $|J'|=22$, Claim 4 shows that $J'_e$ is equivalent to a set $J''_e$ obtained from one of the possibilities for $J_e$ in Table \ref{table:jab=5} by removing a single element. Each set $J_e$ in that table having 8 elements contains two disjoint pairs $i_1,i_2$ and $i_1',i_2'$ at distance 2 (consider pairs $5,7$, $2,4$ and $8,10$) and also contains two disjoint pairs at distance 8 (consider pairs $2,10$, $5,13$ and $5,14$). Thus, after removing a single element, one of each type of those pairs is still in $J''_e$.
    Thus, by Claim 5, $J'_e$ satisfies the conclusion of the claim.
    
    Finally, suppose that $|J_a\cap J_b|=4$. In this case, we must have $J_a\cup J_b\cup J_e = [17]$, thus any element that is not in $J_a\cup J_b$ must be in $J_e$. Using symmetries (including $\phi_4$), one can easily see that there are 8 equivalence classes for having four elements in $J_a\cap J_b$.
    They are listed in Table \ref{table:jab=4}.
    This part is left to the reader.
    By Claim 5, it is sufficient to prove the claim for these $8$ cases
    in order to show the result for any cases involving $|J_a\cap J_b|=4$.
    The table contains additional information if there are two elements $i_1,i_2$ that are forced to be in $J_e$ and have circular distance 2, and similarly for $i_3,i_4$. It remains to find $i_1,i_2$ or $i_3,i_4$ for the cases that are left blank in the table.

    \begin{table}[ht]
    \centering
    \begin{tabular}{|c|c|c|c|l|l|}
        \hline
        Case & $J_a\cap J_b$ & $i_1,i_2$ & $i_3,i_4$ & Candidates for $J_a\setminus J_b$ & Candidates for $J_b\setminus J_a$ \\ 
        \hline
        4.1 & 1 3 6 8 & 2 4 & ~ & 10--16 & 9 11 13 15 17 \\ 
        4.2 & 1 3 8 10 & ~ & ~ & 5 6 12--16 & 2 9 11 13 15 17 \\ 
        4.3 & 1 3 9 11 & ~ & ~ & 5--7 13--16 & 2 4 6 8 10 12 17 \\ 
        4.4 & 1 3 6 9 & 5 7 & 2 10 & 11--16 & 4 8 11 12 15 17 \\ 
        4.5 & 1 3 8 11 & ~ & 4 12 & 5 6 13--16 & 2 6 9 10 13 17 \\ 
        4.6 & 1 3 8 13 & 12 14 & 4 12 & 5 6 10 11 15 16 & 2 6 10 11 15 \\ 
        4.7 & 1 3 6 13 & 5 7 & 5 14 & 8--11 15 16 & 4 8 11 12 15 \\ 
        4.8 & 1 3 6 15 & 5 7 & 7 16 & 8--13 & 4 8 9 12 13 17 \\ 
        \hline
    \end{tabular}
    \caption{Possible cases where $j_0=21$. Those cases are indicated that already have $i_1,i_2$ with $d(i_1,i_2)=2$ and $i_1,i_2\notin J_a\cup J_b$. Similarly for $i_3,i_4$ at circular distance 8.}\label{table:jab=4}
\end{table}

    Let us start with 4.5.
    Notice that $4,7,12 \in J_e$.
    If $2\notin J_b$, then $(i_1,i_2)=(2,4)$ works. Thus, we may assume that $2\in J_b$. But then $6\notin J_b$ since $d(2,6)=4$. Again, we obtain $(i_1,i_2)=(4,6)$, so we may assume that $6\in J_a$. But then $5,7$ cannot belong to $J_a$, and since they also cannot be in $J_b$, they must be in $J_e$. Thus, $(i_1,i_2)=(5,7)$ works.

    We will discuss the remaining cases with fewer details. 
    For 4.1 we consider pairs (2,10), (4,12), (5,14), (7,16) and obtain $i_3,i_4$ as desired or conclude that 12, 14, 16 all belong to $J_a$. In the latter case, we consider (2,11) and (7,15). Since at most one of $11,15$ is in $J_b$ and none is in $J_a$, one of these pairs gives the desired pair $(i_3,i_4)$.

    For 4.2, (4,6) or (5,7) gives $(i_1,i_2)$. Next, we consider (4,12), (4,13), and (9,17). If the first fails, we have $12\in J_a$. If the second fails, we get $13\in J_b$. But now the last pair works. 

    The last case, 4.3, is the longest. We first assume that $5\in J_e$. Then (5,13) or (5,14) gives $(i_3,i_4)$. If there is no $(i_1,i_2)$ at distance 2, consider pairs (5,7), (13,15), (14,16). They imply that $7,13,16\in J_a$ and $14,15\in J_e$. Then $12,17\in J_b$, hence $4,8\in J_e$. Since $d(6,10)=4$, one of these two numbers is in $J_e$, and together with 8, we get the desired pair $(i_1,i_2)$. 
    
    Now suppose in 4.3 that $5 \notin J_e$. Then $5 \in J_a$ since $1 \in J_a \cap J_b$.
    Hence, $6 \notin J_a$ implying that $6 \in J_b$ or (exclusive) $6 \in J_e$.
    If $6 \in J_e$ then (6,14) or (6,15) gives $(i_3,i_4)$.
    Similarly, (4,6) or (6,8) gives $(i_1,i_2)$.
    Otherwise, $6 \in J_b$ implying $2 \in J_e$.
    In this case, (2,17) or (2,4) gives $(i_1,i_2)$.
    To get $(i_3,i_4)$ consider (2,10); since $9 \in J_a\cap J_b$ and $6 \in J_b$, we conclude $10 \in J_e$.
    This completes the proof of the claim.
    \hfill $\diamond$

\bigskip
The same conclusions as in the above claims hold for $J_f$ and for $j_1$ (by replacing the notion of left-maximal with its natural analogue right-maximal). Here we apply the symmetry $\phi_2: [17]\to [17]$ given by the mapping $i\mapsto 2i \pmod{17}$ which corresponds to an automorphism that maps $A$ to $C$, $B$ to $D$, and exchanges $E$ and $F$. 

To complete the proof, assume for a contradiction that a maximum independent set $J\in\cI(\BS)$ has size at least 44. Then we may assume $J_0=J\cap V(\BS_0)$ is left-maximal and $J_1=J\setminus J_0$ is right-maximal. If $|J|\geq 45$ then by Claim 2 we may assume without loss of generality that $|J_0|=23$ and $j_{ab}=6$, implying by Claim 3 that $j_e=10$. But then Claims 3 and 4 imply that $j_f\geq 8$, contradicting (3). Hence we may assume $|J|=44$.

We see that there are only two ways to obtain an independent set of size 44, up to symmetries. We either have $j_0=23$ and $j_1=21$, or (using Claim 4 and (3)) we have $j_0=j_1=22$ with $j_e<10$ and $j_f<10$.


First consider the case where $j_0=23$ and $j_1=21$.
By Claim 3, $$J_e\sim \{2,4,5,7,9,10,12, 14,16,17\}.$$

First suppose that there is a dihedral symmetry $\beta$ such that 
$$J_e = \beta(\{2,4,5,7,9,10,12, 14,16,17\}).$$
Given $J_e \cap J_f = \emptyset$, per (1),
it follows that $J_f \subseteq [17]\setminus \beta(\{2,4,5,7,9,10,12, 14,16,17\})$.
Thus, $J_f \subseteq \beta(\{1,3,6,8,11,13,15\})$.
From here the reader can easily verify $J_f$ cannot have two elements $i$, $j$ such that $d(i,j)=1$,
given $\beta$ preserves circular distance.

Now suppose that 
there exists a dihedral symmetry $\beta$ such that 
$$J_e = \beta\circ \phi_4(\{2,4,5,7,9,10,12, 14,16,17\}).$$
(Here it is possible that one should take $\phi_4 \circ \beta$, or $(\phi_4)^{-1} \circ \beta$, or $(\phi_4)^{-1} \circ \beta$, instead of $\beta\circ \phi_4$, but this has no significant effect on the following argument.)
Given $J_e \cap J_f = \emptyset$, per (1),
it follows that $J_f \subseteq [17]\setminus \beta\circ \phi_4(\{2,4,5,7,9,10,12, 14,16,17\}).$
Thus, $$J_f \subseteq \beta(\{1,4,7,9,10,12,15\}).$$
From here the reader can easily verify $J_f$ cannot have two elements $i$, $j$ such that $d(i,j)=4$,
given $\beta$ preserves circular distance.

Since $j_0 = 23$ we know by (4) that $j_{ab} \geq 6$.
Hence, by Claim 2 $j_{ab} = 6$.
Thus, by Claim 3 we have $j_e=10$, therefore we may assume that $j_f\le7$ by (3).
Since $j_1 = 21$, we may apply Claim~6 to the preimage of $I \cap (F\cup C\cup D)$ under $\phi_2$, concluding $J' = (\phi_2)^{-1}(J_f)$
contains two pairs $i_1,i_2$ and $i_3,i_4$ such that $d(i_1, i_2) = 2$ and $d(i_3, i_4) = 8$.
But this implies $d(\phi_2(i_1),\phi_2(i_2)) = d(2i_1,2i_2) = 4$
and $d(\phi_2(i_3),\phi_2(i_4)) = d(2i_3,2i_4) = 1$
since we are working modulo $17$.
Note that one can take $\phi_8$ rather than $\phi_2$ to achieve an analogous conclusion.
Thus, $J_f$ contains a pair of elements at circular distance $1$, and another pair of elements at circular distance $4$.
But this is a contradiction, since we have shown $J_f$ must either contain no pair of elements at circular distance $1$, 
or $J_f$ must contain no pair of elements at circular distance $4$.
We conclude that $j_0=23$ and $j_1=21$ is impossible.

Finally we consider the case in which $j_0=j_1=22$ with $j_e<10$ and $j_f<10$. Then we use the classification of possible cases given in Claim 4. 
A short computer program was used to confirm that for any two cases $J_e,J'_e$ from Table \ref{table:jab=5}, when we compare $J_e$ for the first one and the right version $J_f=\phi_2(J'_e)$, the two sets intersect. Moreover, they have nonempty intersection even if we apply any equivalence symmetry on $J_f$. This property violates (1) and shows that we cannot obtain an independent set of size 44 in this way.
This completes the proof of the proposition.
\end{proof}

It is indeed interesting to see how tight is the requirement that all cases lead to a contradiction since $J_e$ and $J_f$ cannot be fitted disjointly using any equivalences. 

\end{document}